\documentclass{article}

\usepackage[margin= 0.6in]{geometry}
\usepackage{amsmath,amsthm,amssymb,bm,graphicx }
\usepackage{subfig}
\usepackage{fancyvrb}
\usepackage{mathtools}
\usepackage{hyperref,url,xcolor, soul}
\sethlcolor{pink}
\usepackage{tikz}
\usepackage{comment}
\usetikzlibrary{arrows,matrix,positioning}

\newcommand{\N}{\mathcal N}
\newcommand{\C}{\mathbb C}
\newcommand{\s}{\text{span }}

\newcommand{\adj}{\text{adj}}

\newtheorem{theorem}{Theorem}
\newtheorem{observation}{Observation}
\newtheorem{lemma}{Lemma}

\newtheorem{corollary}{Corollary}

\newtheorem{definition}{Definition}

\hypersetup{
	colorlinks=true,
	linkcolor=blue,%
	citecolor=blue,
	linktoc=page}

\title{Classification of Certain Regular Subalgebras of $\mathfrak{sl}(n, \C)$ up to Conjugacy}
\author{By - Shreya Dhar}
\date{June 3, 2024}
\begin{document}
\maketitle
\begin{abstract}
    In this paper we will be classifying some regular upper-triangular subalgebras  of $\mathfrak{sl}(n, \C)$ up to conjugacy by matrices in $\text{SL}(n, \C)$. We do so for dimension 2, codimension 1, and codimension 2 subalgebras. We prove some general results for codimension $k$. The approach we use reduces an abstract classification problem to a combinatorial one, which we solve through a mixture of inductive and computational approaches.
\end{abstract} 
\section{Notation}
We define some notation.\\ Let $\tau_{ij}$ refer to the standard transposition matrix. Let $n \in \N$. Let $E_{ij}$ and $H_i = E_{i, i} - E_{i+1,i+1} $ be the standard basis of $\mathfrak{sl(n, \mathbb{C}})$. 
We also define, $H_{pq} = E_{pp} - E_{qq} = $. Let $E = \{E_{ij}: 1 \leq i,j \leq n-1, i < j\}$. This is the the basis of upper-triangular nilpotent elements.\\
Let $H = \{H_i: 1 \leq i \leq n -1\}$. This is the basis of diagonal elements.\\
Let $H' = \{H_{pq}: p < q\}$.\\
A regular subalgebra is one that is centralized by a Cartan subalgebra. In this paper, we shall be looking at regular subalgebras spanned by elements of $E$ and $H$. 
\begin{definition} 
A star matrix is a visual way to represent the nilpotent basis of the subalgebra $S$ of $\mathfrak{sl}(n, \C)$. In the $n \times n$ matrix, we have a $\ast$ in the $ij$th position for every $E_{ij}$ in the basis, and $0$ when it is not in the basis. \\
Additionally, for solvable algebras, we will have linear combinations of the $H_i$ in the diagonal.
\label{def: regular subalgebra and star matrix}
\end{definition} 
The star matrix represents an arbitrary element in the subalgebra; a star indicates that this position in the matrix can be non-zero. Once we look at right/left action of the algebra on $\C^n$ and column/row rank, we shall calculate those using the star matrix. For example, $    \begin{bmatrix}
        0 & 0 & \ast & \ast\\
        0 & 0 & 0 & 0\\
        0 & 0 & 0 & \ast\\
        0 & 0 & 0 & 0
    \end{bmatrix} $ 
is the star matrix of a purely nilpotent subalgebra with basis $E_{13}, E_{14}, E_{34}$ in $\mathfrak{sl}(4, \C)$. $\begin{bmatrix}
        a & 0 & \ast & \ast\\
        0 & b-a & 0 & 0\\
        0 & 0 & -b & 0\\
        0 & 0 & 0 & 0
    \end{bmatrix} $
This is an example of a solvable subalgebra of $\mathfrak{sl}(4, \C)$, with basis $E_{13}, E_{14}, H_1, H_2$. We may write a star matrix as follows:
\begin{equation*}
    \begin{bmatrix}
    0 & \ast_{12} & \ast_{13} & \ast_{14} \\
    0 & 0 & 0 & 0\\
    0 & 0 & 0 & 0\\
    0 & 0 & 0 & 0\\    
    \end{bmatrix} \cdot \begin{bmatrix}
    \ast_1 \\ \ast_2 \\ \vdots \\ 0 
    \end{bmatrix} = \begin{bmatrix}
    \ast_1 \\ 0 \\ 0 \\ 0
    \end{bmatrix} 
\end{equation*}
$\ast$ is a potential non-zero entry (as the star matrix represents all the possible elements of the algebra) and the sub-script indicates which row/column it is in. 
\begin{definition}
    We define ${\C^n}_{i} =  \begin{bmatrix}
        \ast_1 \\ \vdots \\ \ast_{i} \\ 0  \\ \vdots \\ 0 
    \end{bmatrix}$ as the column vector where the entries $\ast_i \in \C$. 
\end{definition}

\section{General Theorems, Lemmas, and Definitions }
\subsection{Lie Bracket}
The general Lie bracket is as follows:
\begin{equation}
    [E_{ij}, E_{kl}] = E_{ij}E_{kl} - E_{kl}E_{ij} = \delta_{jk}E_{il} - \delta_{li}E_{kj}
    \label{Lie 1}
\end{equation}
\begin{equation*}
    [E_{ij}, H_k] = E_{ij}E_{kk} - E_{ij}E_{k+1,k+1} - E_{kk}E_{ij} + E_{k+1,k+1}E_{ij} = \delta_{jk}E_{ik} - \delta_{j, k+1}E_{i,k+1} - \delta_{ki}E_{kj} + \delta_{k+1,i}E_{k+1,j}
    \label{Lie 2}
\end{equation*}
\begin{equation*}
    [H_i, H_k] = 0
    \label{Lie 3}
\end{equation*}
The last equation follows from the fact that diagonal matrices commute. 

\begin{observation} Let the star matrix $N$ represent the maximal nilpotent subalgebra of a regular-algebra, i.e. it is spanned by the set of all $E_{ij}$'s in the sub-algebra. Then $N^2$ represents the commutator of the subalgebra. Further, $N^{2^k}$ represents the $k^{th}$ element in the derived series. 
\label{multiplication by star matrix}
\end{observation}
\subsection{Lemmas on closure under Lie bracket:}
\begin{lemma}
(Nilpotent Element Removal) Fix $i < j$. Let $L$ be a upper-triangular regular subalgebra with the standard basis, without $E_{ij}$ as a basis element. Then for $k \geq i + 1$ and $k \leq j - 1$, $E_{ik}$ is not a basis element in the subalgebra or $E_{kj}$ is not a basis element in the subalgebra. \\
Representing $A$ as a star matrix, this means that the ${ik}^{th}$ entry is 0 or the ${kj}^{th}$ entry is 0, i.e., each $A_{ik}A_{kj} = 0$. That is either an element to the left or the corresponding element below has to be 0. Then, $dim L \leq \frac{n(n+1)}{2} - j + i$. If $E_{ij}$ is not a basis element in a nilpotent subalgebra, $N$ then $dim N \leq \frac{n(n-1)}{2} - j + i$. This is a tight bound.
\label{lemma 1}
\end{lemma} 
\begin{proof}
Suppose $[E_{mn}, E_{kl}] = E_{ij}$. This implies that $E_{ij} = E_{ml}$, and so $i = m$ and $l = j$ where $n = k$. Thus, $E_{ij} = [E_{ik}, E_{kj}]$ where $k \in \{1, ..., n\}$. As this is a regular subalgebra, $k \leq j - 1$ and $k \geq i + 1$. As $L$ is a subalgebra, one of the elements in these pairs is not in the Lie algebra, otherwise $L$ would not be closed under the Lie bracket. This is precisely our condition. Therefore, at least $j - i -1$ elements from $|H|$ and $|E|$ are not in the subalgebra. Since $|H| + |E| = \frac{n(n+1)}{2} $, $dim L \leq \frac{n(n+1)}{2} - j + i$. As $|E| = \frac{n(n-1)}{2} $ if $E_{ij}$ is not a basis element in a nilpotent subalgebra, $N$ then $\dim N \leq \frac{n(n-1)}{2} - j + i$. 
This is the lowest upper bound. If we remove all of the $E_{ik}$ for $k \leq j-1$ and $k \geq i+1$ (or similarly, all the $E_{kj}$) from the set of $H \cup E$, i.e., let $L = \s{E \cup H \setminus \{E_{ik}: k \in \{i + 1,, ..., j\}\}}$ then L is a subalgebra. For any $E_{ik}$ for $k \in \{i+1, ..., j\}$ none of the $E_{ik}$ are in the subalgebra, and so for $I_k = \{i+1, ...,k\}$, none of the $E_{il}$ for $l \in I_k$ are in the subalgebra, and so $\dim L = \frac{n(n+1)}{2} - j + i$. Therefore the subalgebra is closed under the Lie bracket and the bound is tight.
\end{proof}
\begin{corollary}
(Removal of Off-Diagonal) Let $M$ be a subalgebra. If we remove the basis element $E_{i,i+1}$, call the resulting subalgebra $P$. Then $\dim P = \dim M - 1$. 
\end{corollary}

\begin{lemma}(Removal or Addition of $H_i$) Let $L$ be a subalgebra with standard basis. Let $P$ be the set of basis vectors. Then,
$L_R = \s{P \setminus \{H_i\}}$ and $L_W = \s{P \cup \{H_i\}}$ are both subalgebras. \label{remove h_i}
\end{lemma}
  
\subsection{Lemmas on Invariance}
Let $L$ and $M$ be regular upper-triangular sub-algebras of $\mathfrak{sl}(n,\C)$ spanned by $E$ and $H$,  which are conjugate by $g \in SL(n,\C).$ 
\begin{lemma} (Nilpotent Invariance) Nilpotent elements conjugate to nilpotent elements.
\label{nilpotent conjugacy}
\end{lemma}
\begin{corollary}
    For $E_{ij} \in L$ then $g$ conjugates it to a linear combination of $E_{kl} \in M$.
    \label{nilpotent algebras conjugate to nilpotent algebras}
\end{corollary}
\begin{corollary}
    Two conjugate subalgebras have the same dimension for the maximal nilpotent subalgebra.
    \label{nilpotent dimension}
\end{corollary}
\begin{definition}
    We define {nilpotent dimension} as the dimension of the maximal nilpotent subalgebra of a Lie algebra.
\end{definition}

\begin{lemma} {(Dimension of Commutator)} $[L, L] $ is conjugate to $[M, M]$ and so the dimension of the commutator is invariant under conjugation.
\label{dimension of commutator}
\end{lemma}

\begin{corollary} (Derived Series under Conjugacy)
    The derived series of $L$ conjugates to the derived series of $M$. Thus, the dimension of the derived series is invariant under conjugation.
    \label{dimension of derived series}
\end{corollary}

\begin{definition}
    The max rank of a Lie algebra is the maximum rank (the column/row rank of a matrix) that an element in the algebra can obtain. Similarly, the min rank is the minimum rank that a non-zero element of the algebra can obtain.
\end{definition}
\begin{lemma}
    The min and max rank are invariant under conjugation.
    \label{min/max rank invariance}
\end{lemma}
\begin{lemma}(Left action on column vectors)
\label{acting on column vectors lemma} Two conjugate subalgebras $L,M$ send $\mathbb{C}^n$ to subspaces of the same dimension. That is, when a star matrix acts on a column vector with entries in $\mathbb{C}$, the number of non-zero entries in $Av$ is the same as the number of non-zero entries in $Bv$ where $A$ and $B$ are the star matrices for $L,M$ respectively.
\end{lemma}
\begin{lemma}(Right action on row vectors)
\label{acting on row vectors lemma} Two conjugate subalgebras $L,M$ have the same right action on row vectors. 
\end{lemma}
\begin{definition}
    We define \emph{column rank} as the dimension of subspace that the left action of the Lie algebra sends $\C^n$ to. Similarly, we define \emph{row rank} as the dimension of subspace that the right action of the Lie algebra sends $\C^n$ to.  
\end{definition}
\begin{corollary}
    The column/row rank of the derived series are equal, i.e., the column/row ranks of ${L}^i$ is the same as the column/row ranks of ${m}^i$ (here $i$ refers to the $i^{th}$ term in the derived series.)
    \label{Derived Series Left and row rank}
\end{corollary}
\subsection{Lemmas on Conjugacy}
Let the conjugation matrix be $X = [x_{i}]_{i = 1}^{n + n(n-1)}$ with inverse $X^{-1} = [a_{ij}]_{n \times n} \implies XE_{kl}X^{-1} = [x_{k + (i-1)n}a_{lj}]_{i,j = 1}^{n}$. \\Similarly, $XH_{pq}X^{-1} = [x_{p + (i-1)n}a_{pj} - x_{q + (i-1)n}a_{qj}]_{i,j = 1}^n$
\begin{lemma}(Conjugacy for $E_{ij}$) 
For $E_{ij}$ to be strictly upper triangular, the last row must be 0, and if all the $a_{il}$ for $1 \leq l \leq n$ are 0, then the final matrix is 0. So, to conjugate $E_{ij}$ to be strictly upper triangular and non-zero, we need $x_{i + n(n-1)} = 0$.
\end{lemma}
\begin{proof}
 For the last row to be zero but the matrix non-zero, $x_{k + n(n_1)} = 0$.
\end{proof}
\begin{lemma}(Conjugacy of $H_{pq}$): Let $L, M$ be regular upper triangular sub-algebras spanned by $E$ and $H'$. Let $H_{an}$ be an element of $L$ and let an arbitrary number of $H_{pq}$'s be basis elements for $M$ for $q < n$. Suppose $L$ does not have any of these $H_{pq}'s$'s as basis elements and $M$ does not have $H_{an}$ as a basis element. Suppose $M$ also has $E_{pj}$ and $E_{q, k}$ as basis elements for some $j$ and $k$. Then $L$ and $M$ are not conjugate.
\label{conjugacy of h_pq}
\end{lemma}
\begin{proof}
We know that we can always conjugate the $E_{ij}$ to be strictly upper triangular as they are nilpotent, and thus only conjugate to nilpotent elements (thus, none of the $E_{ij}$ conjugate to $H_{n-1}$). As we have $E_{pj}$ and $E_{q,k}$ for some $j$ and $k$,as seen in the conjugation above, we need $x_{p + n(n-1)} = x_{q + n(n-1)} = 0$. Now, note for any $H_{pq}$, every term in the last row is of the form:
$x_{p + n(n-1)}a_{i,...} - x_{q + n(n-1)}a_{i+1,..}$.
However, $x_{p + n(n-1)} = x_{q +  n(n-1)} = 0$. Therefore, the last row for $H_{pq}$ is 0. As $H_{an} = E_{a,a} - E_{n,n}$ and no other basis element of a solvable subalgebra has the term $E_{n,n}$, no linear combination of $H_{pq}$'s in $L_2$ conjugates to $H_{an}$. Therefore, $L_1$ and $L_2$ are not conjugate. 
\end{proof}

\section{Adjoint Action of $H_i$}
We first discuss a few useful facts about elements $H_i$ can be conjugated to.
\begin{lemma}
    \label{conjugation of center} If $L$ and $M$ are two conjugate Lie algebras then the center of $L$ conjugates to the center of $M$.
\end{lemma}
\begin{corollary}
    Let $L = \s(E \cup \{H_{i_1}, ..., H_{i_k}\})$ and $M$ be a conjugate algebra by the matrix $A$. Then $AH_{i_1}A^{-1} = \sum_{j=1}^{n-1}a_jH_j$.
    \label{full E, H center}
\end{corollary}
\begin{proof}
    As the L has full nilpotent basis, so does $M$. Thus, $C(M) \subset \s(H_1,..., H_{n-1}\})$. 
\end{proof}
\begin{corollary}
    Let $L = \s(E -E_{m,m+1} \cup \{H_{i_1}, ..., H_{i_k}\})$ and $M$ be a conjugate algebra by the matrix $A$. Then $AH_{i_1}A^{-1} = \sum_{j=1}^{n-1}a_jH_j$.
    \label{one less than full E}
\end{corollary}
\begin{proof}
    As L has codimension 1 nilpotent basis, so does $M$. Thus, $C(M) \subset \s(H_1,..., H_{n-1})$. 
\end{proof}

\begin{lemma}
    In the case where the nilpotent basis of the algebra is $E$ or $E \setminus \{E_{i,i+1}\}$, $H_k$ can only conjugate to elements of the form $H_{ij} = E_{ii} - E_{jj}$
    \label{H_i exact conjugation}
\end{lemma}
\begin{proof}
    Let $P \in \s(H_1,.., H_{n-1})$ such that $P = \text{diag}(a_1, a_2 - a_1, \hdots a_{n-1} - a_{n-2}, -a_{n-1})$ with characteristic polynomial $f_P(x) = (x - a_1)(x -(a_2 - a_1)) \cdot ... \cdot (x - (a_{n-1} - a_{n-2})(x+a_{n-1})$. Recall, the characteristic polynomial is invariant under conjugation, and the characteristic polynomial of $H_i$ is $f_{H_i}(x) = (x-1)(x+1)x^{n-2}$. So, $n-2$ of the entries in $P$ should be zero. If either $a_1$ or $a_{n-1}$ is zero, that would force at least $n-1$ of the entries, and so neither of those two can be zero. Thus, $a_1 = a_{n-1} = +/-1$ and the rest are zero. If $A$ conjugates $H_i$ to $P$ such that $a_1 = -1$, $-P = -AH_i-A^{-1}$ is a matrix whose first entry is 1, so WLOG, let $a_1 = a_{n-1} = 1$ (because any algebra containing $E_{ii} - E_{jj} = P$ also contains $-P$).
\end{proof}

When we look for algebras conjugate to those with codimension 0 or 1 nilpotent basis, it is enough to check whether $H_k$ can conjugate to any other $H_{ij}$. We now state a theorem that we shall prove in section 6.1:
\begin{theorem} (Complete Classification of $\frac{n(n-1)}{2} - 1$ nilpotent upper-triangular regular subalgebras)
    Let the maximal nilpotent subalgebra of each $L_{i,i+1}$ be denoted by $N_{i,i+1} = \{E_{ij} \in L_{i,i+1}\}$. Then $\{N_{i,i+1}\}$ are all in distinct conjugacy classes.
\end{theorem}
In particular, if we are considering subalgebras of the following two types: $L_{pq} = \s(E \cup H \setminus \{H_{p_1,q_1},..., H_{p_t,q_t}\})$ and $ L'_{i,pq} = \s(E \setminus \{E_{i,i+1}\} \cup H \setminus \{H_{p_1,q_1},..., H_{p_t,q_t}\})$. It is sufficient to fix the nilpotent basis, as $\s(E - E_{j,j+1})$ cannot conjugate to any $\s(E - E_{m,m+1})$ for $j \neq m$. For the first case, we have full nilpotent basis so it conjugates to full nilpotent basis. For the second case, we fix $i$ and drop the $i$ in the notation, and refer to those algebras as $L'_{pq}$. 

\subsection{Adjoint Action of $H_{pq}$ on $\s(E)$}
Let $\adj(H_k)(\s(E)) = A_k$ and $\adj(H_{pq})(\s(E)) = A_{pq}$. Recall that $[H_k, E_ij] \neq 0 \iff i = k$ or $k + 1$ or $j = k$ or $k + 1$. See table \ref{Lie 2}. 
\begin{equation*}
    \adj(H_k)(\s(E)) = \begin{bmatrix}
        0_{11} & & & \ast_{1k} & \ast_{1,k+1} &  \\
        & \ddots &   & \vdots & \vdots & \\
        &  &   0_{k-1,k-1} & \ast_{k-1,k} & \ast_{k-1,k+1} & \\
         & & & 0_{k,k} & \ast_{k,k+1} & \ast_{k,k+2} & \hdots & \ast_{k,n} \\
         &  & & & 0_{k+1,k+1} & \ast_{k+1,k+2} & \hdots &\ast_{k+1, n} \\
          &  & & & &  \ddots  \\
           &  & & & & & \ddots  \\
          &  & & & & & &0_{nn}  \\
    \end{bmatrix}
\end{equation*}

\begin{equation*} A_{pq} = \adj(H_{pq})(\s(E)) = 
    \left[ \begin{array}{@{}*{15}{c}@{}}
      0_{11} & & & \ast_{1p} & 0_{1,p+1} & \hdots & 0_{1,q-1} & \ast_{1,q} &  \\
        & \ddots &   & \vdots & \vdots &  & \vdots & \vdots \\
        &  &   0_{p-1,p-1} & \ast_{p-1,p} & 0_{p-1,p+1} & \hdots & 0_{p-1,q-1} & \ast_{p-1,q} & \\
        &  &   & 0_{p,p} & \ast_{p,p+1} & \hdots  & \hdots & \hdots & \hdots & \hdots & \ast_{p-1,n} \\
        &  &   &  &\ddots   &  &  & \vdots & \\
        &  &   & & & \ddots  &  & \vdots & \\
        & & & & &   & 0_{q-1,q-1} & \ast_{q-1,q}  \\
        & & & & &   &  & 0_{q,q} &\ast_{q,q+1} & \hdots & \ast_{q,n} \\
        & & & & &   &  &  &\ddots &  &  \\
        & & & & &   &  &  & & \ddots   &  \\
                & & & & &   &  &  & &    & 0_{nn}  \\

    \end{array} \right]
\end{equation*}
Examining the matrices above, we get, $\C^n A_{pq} = \C^{n - (p-1)}$. However, if $H_{pq} \in L_1 = \s(E \cup \{H_{p_1,q_1},..., H_{p_k,q_k}\})$ conjugates to $H_{ab} \in L_2 = \s(E \cup \{H_{a_1,b_1},..., H_{a_k,b_k}\})$, then $A_{pq}$ and $A_{ab}$ are conjugate as matrices. This would mean they have the same action on row vectors, i.e. $\C^n PA_{pq}P^{-1} = \C^n A_{pq}  P^{-1}  {\mathbb C}^n$. ${\mathbb C}^n P = \mathbb{C}^n$ as $P$ is an invertible matrix, and so $ \C^nA_{pq}P^{-1}$ is the same dimension as $\C^n A_{pq}$. We get:
\begin{lemma}
    $H_{pq} \in L_1 = \s(E \cup \{H_{p_1,q_1},..., H_{p_k,q_k}\})$ conjugates to $H_{ab} \in L_2 = \s(E \cup \{ H_{a_1,b_1},..., H_{a_k,b_k}\})$ $\iff$ $p = a$. 
    \label{conjugacy of H for full nilpotent basis}
\end{lemma}
We need only distinguish $H_{k}$ from $H_{k,q}$. We treat the case $q = n$ separately. Suppose $q < n$, $A_{pq} \C^n = \C^{q}$. If $H_{pq} \in L_1 = \s(E \cup \{H_{p_1,q_1},..., H_{p_k,q_k}\})$ conjugates to $H_{pb} \in L_2 = \s(E  \cup \{H_{p_1,b_1},..., H_{p_k,b_k}\})$, then $A_{pq}$ and $A_{pb}$ are conjugate as matrices. This would result in the same column action $PA_{pq}P^{-1} \C^n  = PA_{pq} \C^n$. $ P^{-1}\C^n = \mathbb{C}^n$ as $P$ is an invertible matrix, and so $ PA_{pq}\C^n$ is the same dimension as $ A_{pq}\C^n$. Finally $H_{p,n}$ cannot conjugate to any $H_{pq}$ for $q \leq n-1$ by lemma \ref{conjugacy of h_pq}. 
\begin{theorem} (Conjugacy of $L_{pq}$)
    Let $L_{pq} = \s(E \cup H \setminus \{H_{p_1,q_1},..., H_{p_t,q_t}\})$. $L_{pq}$ conjugates to $L_{ab} \iff (p,q) = (a,b)$
    \label{conjugacy of l_ab}
\end{theorem}
\begin{proof}
    If $a \neq p$, they we have different action on row vectors. If $q \neq b$, we have different action on column vectors for $q < n$, and no $q < n$ conjugates to $q = n$ by lemma \ref{conjugacy of h_pq}.
\end{proof}

\subsection{Adjoint Action of $H_{pq}$ on $\s(E \setminus \{ E_{i,i+1}) \}$}
 We analyze $\adj(H_{pq})(E \setminus \{ E_{i,i+1}\})$. Examining the matrices in tables \ref{tab:adjoint action of h_Pq on one less than full, cases 1-3}, \ref{tab:adjoint action of h_Pq on one less than full, cases 4-7}, and \ref{tab:adjoint action of h_Pq on one less than full, case 8} we have: $\C^n A_{pq} = \C^{n - (p-1)}$. If $H_{pq} \in L_1 = \s(E \setminus \{E_{i,i+1}\} \cup \{H_{p_1,q_1},..., H_{p_k,q_k}\})$ conjugates to $H_{ab} \in L_2 = \s(E \setminus \{E_{i,i+1}\} \cup \{H_{a_1,b_1},..., H_{a_k,b_k}\})$, then $A_{pq}$ and $A_{ab}$ are conjugate as matrices. This should result in same action on row vectors i.e. $
\C^n PA_{pq}P^{-1} = \C^n A_{pq}  P^{-1}  {\mathbb C}^n$. So: 
\begin{lemma}
    $H_{pq} \in L_1 = \s(E \setminus \{E_{i,i+1}\} \cup \{H_{p_1,q_1},..., H_{p_k,q_k}\})$ conjugates to $H_{ab} \in L_2 = \s(E \setminus \{E_{i,i+1}\} \cup \{ H_{a_1,b_1},..., H_{a_k,b_k}\})$ $\iff$ $p = a$. 
    \label{conjugacy when p is not equal to a}
\end{lemma}
Thus, we need only distinguish $H_{k}$ from $H_{k,q}$. We shall treat the case $q = n, n-1$ separately. Suppose $q \leq n -2$, i.e. $q +2 \leq n$. For all cases: $A_{pq} \C^n = \C^{q}$. If $H_{pq} \in L_1 = \s(E \setminus \{E_{i,i+1}\} \cup \{H_{p_1,q_1},..., H_{p_k,q_k}\})$ conjugates to $H_{pb} \in L_2 = \s(E \setminus \{E_{i,i+1}\} \cup \{H_{p_1,b_1},..., H_{p_k,b_k}\})$, then $A_{pq}$ and $A_{pb}$ are conjugate as matrices. This results in same column action, i.e. $ PA_{pq}P^{-1} \C^n  = PA_{pq} \C^n $. Thus, we have:
\begin{lemma} For $q_1,...,q_k \leq n -2 $
    $H_{pq} \in L_1 = \s(E \setminus \{E_{i,i+1}\} \cup \{H_{p_1,q_1},..., H_{p_k,q_k}\})$ conjugates to $H_{ab} \in L_2 = \s(E \setminus \{E_{i,i+1}\} \cup \{ H_{a_1,b_1},..., H_{a_k,b_k}\})$ $\iff$ $(p,q) = (a,b)$  i.e. the two algebras are identical.
    \label{conjugacy for the second case}
\end{lemma}
We must check if any $H_{p,n-1}$ and $H_{p,n}$ conjugate to each other and any $H_{pq}$ for $q < n-1$. $A_{pq}(\C^n) = \C^{q}$. For $q = n-1$, for all cases 1,2, 5-8 (case 3 doesn't occur), we have $A_{p,n-1}(\C^n) = \C^{n-1}$. Thus $H_{pq}$ doesn't conjugate to $H_{p,n-1}$ when $i \neq n-1$, as they have different action on column vectors. For $i = n-1$ i.e. we are looking at nilpotent basis $E \setminus \{E_{n-1,n}\}$, we have $A_{p,n-1}(\C^n) = \C^{n-2}$. Here, $H_{p,n-1}$ is not conjugate to any $H_{pq}$ for $q < n-2$. We must distinguish $H_{p,n-2}$ and $H_{p,n-1}$.  $H_{p,n-2}$ belongs in case 7 which has max rank 4 and $H_{p,n-1}$ belongs in case 4 which has max rank 3, thus they are not conjugate by lemma \ref{min max rank of adjoint action}.
\begin{table}[h!]
    \centering
    \begin{tabular}{||c | c|  c||} 
 \hline
 No. & Case & $\adj(H_{pq})(E \setminus \{E_{i,i+1}\})$\\ [0.5ex] 
 \hline\hline
 1 & $i, i + 1 < p < q$ & $
    \left[ \begin{array}{@{}*{15}{c}@{}}
       0_{11} & & & \ast_{1p} & 0_{1,p+1} & \hdots & 0_{1,q-1} & \ast_{1,q} &  \\
        & \ddots &   & \vdots & \vdots &  & \vdots & \vdots \\
        &  &   0_{p-1,p-1} & \ast_{p-1,p} & 0_{p-1,p+1} & \hdots & 0_{p-1,q-1} & \ast_{p-1,q} & \\
        &  &   & 0_{p,p} & \ast_{p,p+1} & \hdots  & \hdots & \hdots & \hdots & \hdots & \ast_{p-1,n} \\
        &  &   &  &\ddots   &  &  & \vdots & \\
        &  &   & & & \ddots  &  & \vdots & \\
        & & & & &   & 0_{q-1,q-1} & \ast_{q-1,q}  \\
        & & & & &   &  & 0_{q,q} &\ast_{q,q+1} & \hdots & \ast_{q,n} \\
        & & & & &   &  &  &\ddots &  &  \\
        & & & & &   &  &  & & \ddots   &  \\
                & & & & &   &  &  & &    & 0_{nn}  
    \end{array} \right] $\\ 
 \hline
 2 & $i = p < i + 1 < q$ & $\left[ \begin{array}{@{}*{17}{c}@{}}
     & & \ast_{1,p} & & & & & & & \ast_{1,q} & & &\\
     & & \vdots & & &  & & & & \vdots & & &\\
     & & \ast_{p-1,p} & & & & & & & \ast_{p-1,q} & & &\\
     & & 0_{pp} & \ast_{p,p+1} & \hdots & \ast_{p,i} & 0_{p,i+1} & \ast_{p,i+2} & \hdots & \ast_{pq} & \hdots & \hdots & \ast_{pn}\\
     & & &  &  &  &  & &  & \vdots &  & \\
     & & &  &  &  &  & &  & \ast_{q-1,q} &  & \\
     & & &  &  &  &  & &  & 0_{q,q} & \ast_{q,q+1} & \hdots & \ast_{q,n} \\
     & & &  &  &  &  & &  & &  &  & \\
    & & &  &  &  &  & &  & &  &  & 
\end{array} \right]$
 \\
 \hline
 3 & $i = p < q = i + 1$ & $\begin{bmatrix}
        & & \ast_{1,i} & \ast_{1,i+1} & & & \\
        & & \vdots & \vdots & & & &\\
        & & \ast_{i-1,i} & \ast_{i-1,i+1} & & & \\
        & & 0 & 0 & \ast_{i,i+2}& \hdots & \ast_{i,n}\\
        & & 0 & 0 & \ast_{i+1,i+2}& \hdots & \ast_{i+1,n}\\
        & &  &  & &  & \\
         & &  &  & &  & \\        
    \end{bmatrix}$\\ [1ex]
 \hline 
\end{tabular}
    \caption{Adjoint Action of $H_{pq}$ on $E \setminus \{E_{i,i+1}\}$, Cases 1 -3}
    \label{tab:adjoint action of h_Pq on one less than full, cases 1-3}
\end{table}
\\
\begin{table}[]
    \centering
    \begin{tabular}{||c | c|  c||} 
 \hline
 No. & Case & $\adj(H_{pq})(E \setminus \{E_{i,i+1}\})$\\ [0.5ex] 
 \hline\hline
4 & $p < i = q < i + 1$ &  $\left[ \begin{array}{@{}*{17}{c}@{}} 
 & & \ast_{1,p} &  & &\ast_{1,i} \\
 & & \vdots & & &\vdots \\
 & & \ast_{p-1.p}& & & \ast_{p-1.i} \\
 & & 0_{pp} & \ast_{p,p+1} & \hdots & \ast_{p,p+i -1} & \ast_{p,i} & \ast_{p,i+ 1} & \hdots & \hdots & \ast_{pn}\\
 & &  & &  & & \vdots &  &  & \\
 & &  & &  & & \ast_{i-1,i} &  &  & \\
 & &  & &  & & 0_{ii} & 0_{i,i+1} & \ast_{i,i+2}  & \hdots & \ast_{i,n} \\
 & &  & &  & &  &  &  &  \\
 & \\
 & 
 \end{array} \right]$\\
 \hline
 5 & $p < i , i + 1 < q$ & $\left[ \begin{array}{@{}*{15}{c}@{}}
       0_{11} & & & \ast_{1p} & 0_{1,p+1} & \hdots & 0_{1,q-1} & \ast_{1,q} &  \\
        & \ddots &   & \vdots & \vdots &  & \vdots & \vdots \\
        &  &   0_{p-1,p-1} & \ast_{p-1,p} & 0_{p-1,p+1} & \hdots & 0_{p-1,q-1} & \ast_{p-1,q} & \\
        &  &   & 0_{p,p} & \ast_{p,p+1} & \hdots  & \hdots & \hdots & \hdots & \hdots & \ast_{p-1,n} \\
        &  &   &  &\ddots   &  &  & \vdots & \\
        &  &   & & & \ddots  &  & \vdots & \\
        & & & & &   & 0_{q-1,q-1} & \ast_{q-1,q}  \\
        & & & & &   &  & 0_{q,q} &\ast_{q,q+1} & \hdots & \ast_{q,n} \\
        & & & & &   &  &  &\ddots &  &  \\
        & & & & &   &  &  & & \ddots   &  \\
                & & & & &   &  &  & &    & 0_{nn}  
    \end{array} \right] $ \\
    \hline 
     6 & $p < i < q = i + 1$ & $\left[ \begin{array}{@{}*{15}{c}@{}}
    & & \ast_{1,p} & & & & \ast_{1,i + 1} \\
    & & \vdots & & & & \vdots \\
    & & \ast_{p-1,p} & & & & \ast_{p-1, i + 1}\\
    & & 0_{pp} & \ast_{p,p+1} & \hdots & \ast_{p,i} & \ast_{p,i+1} & \hdots & \hdots & \ast_{pn}\\
    & &  &  & &  & \vdots&  &  & \\
    & &  &  & &  & \ast_{i-1,i+1}&  &  & \\
    &  &  &  & &  & 0_{i,i+1}&  &  & \\
    &  &  &  & &  & 0_{i+1 ,i+1}& \ast_{i+1, i + 2} & \hdots  &\ast_{i_1, n} \\
    & \\
    &    \end{array} \right]$
    \\
    \hline 
    7 & $p < q < i,i + 1 $& $
    \left[ \begin{array}{@{}*{15}{c}@{}}
       0_{11} & & & \ast_{1p} & 0_{1,p+1} & \hdots & 0_{1,q-1} & \ast_{1,q} &  \\
        & \ddots &   & \vdots & \vdots &  & \vdots & \vdots \\
        &  &   0_{p-1,p-1} & \ast_{p-1,p} & 0_{p-1,p+1} & \hdots & 0_{p-1,q-1} & \ast_{p-1,q} & \\
        &  &   & 0_{p,p} & \ast_{p,p+1} & \hdots  & \hdots & \hdots & \hdots & \hdots & \ast_{p-1,n} \\
        &  &   &  &\ddots   &  &  & \vdots & \\
        &  &   & & & \ddots  &  & \vdots & \\
        & & & & &   & 0_{q-1,q-1} & \ast_{q-1,q}  \\
        & & & & &   &  & 0_{q,q} &\ast_{q,q+1} & \hdots & \ast_{q,n} \\
        & & & & &   &  &  &\ddots &  &  \\
        & & & & &   &  &  & & \ddots   &  \\
                & & & & &   &  &  & &    & 0_{nn}  
    \end{array} \right] $ \\ [1ex]
    \hline 
\end{tabular}
    \caption{Adjoint Action of $H_{pq}$ on $E \setminus \{E_{i,i+1}\}$, Cases 4 -7}
    \label{tab:adjoint action of h_Pq on one less than full, cases 4-7}
\end{table}
\\
\begin{table}[h!]
    \centering
    \begin{tabular}{||c | c|  c||} 
 \hline
 No. & Case & $\adj(H_{pq})(E \setminus \{E_{i,i+1}\})$\\ [0.5ex] 
 \hline\hline
 8 & $i < p = i + 1 < q $ & $ \left[ \begin{array}{@{}*{15}{c}@{}}
       0_{11} & & & \ast_{1p} & 0_{1,p+1} & \hdots & 0_{1,q-1} & \ast_{1,q} &  \\
        & \ddots &   & \vdots & \vdots &  & \vdots & \vdots \\
        &  &   0_{p-1,p-1} & \ast_{p-1,p} & 0_{p-1,p+1} & \hdots & 0_{p-1,q-1} & \ast_{p-1,q} & \\
        &  &   & 0_{p,p} & \ast_{p,p+1} & \hdots  & \hdots & \hdots & \hdots & \hdots & \ast_{p-1,n} \\
        &  &   &  &\ddots   &  &  & \vdots & \\
        &  &   & & & \ddots  &  & \vdots & \\
        & & & & &   & 0_{q-1,q-1} & \ast_{q-1,q}  \\
        & & & & &   &  & 0_{q,q} &\ast_{q,q+1} & \hdots & \ast_{q,n} \\
        & & & & &   &  &  &\ddots &  &  \\
        & & & & &   &  &  & & \ddots   &  \\
                & & & & &   &  &  & &    & 0_{nn}  
    \end{array} \right]$\\ [1ex]
 \hline 
\end{tabular}
    \caption{Adjoint Action of $H_{pq}$ on $E \setminus \{E_{i,i+1}\}$, Case 8}
    \label{tab:adjoint action of h_Pq on one less than full, case 8}
\end{table}
\begin{lemma} (Invariance of max rank) Let $M$ be a Lie algebra. Let $x \in M$. Then $\max_{y \in M}(\adj(x)y)$ and  $\min_{y \in M}(\adj(x)y)$ is invariant under conjugation.\label{min max rank of adjoint action}
\end{lemma}
By lemma \ref{conjugacy of h_pq}, $H_{pq}$ doesn't conjugate to $H_{pn}$  when $i \neq n -1$. Similarly, $H_{pn}$ doesn't conjugate to $H_{p,n-1}$ when $i \neq n-1$ (we deal with this case at the end). Finally, in any Lie algebra with nilpotent part  $\s(E \setminus \{E_{n-1,n}\})$, the transposition matrix $\tau_{n-2,n-1}$ conjugates $H_{p,n}$ to $H_{p,n-1}$. 
Hence, we obtain the following classification theorem:
\begin{theorem} (Classification of algebras with nilpotent part $E \setminus \{E_{i,i+1}\}$
    Let $1 \leq i \leq n -2$ and $1 \leq k \leq n - 1$. Define:
\begin{equation*}
    L'_{i,pq} = \s(E \setminus \{E_{i,i+1}\} \cup H \setminus \{H_{p_1q_1},..., H_{p_t, q_t}\})
\end{equation*}
$L'_{i,pq}$ and $L
_{j,ab}$ are conjugate $\iff (i,p,q) = (j,a,b)$ (i.e. all the indices are the same). \\
$L'_{n-1,pq}$ and $L
_{j,ab}$ are conjugate $\iff j = n-1$ and $(p_k,q_k) = (a_k,b_k)$ or $p_k = a_k$, $q_k = n-2$ and $b_k = n -1$.
\end{theorem}

\section{Some Classifications for $\mathfrak{sl}(n, \mathbb{C})$ }
We classify upper-triangular regular sub-algebras of $\mathfrak{sl}(n,\C)$ spanned by some elements of $E$ and $H$.
\subsection{Classification of $\frac{n(n+1)}{2} - 1$ dimensional upper-triangular regular subalgebras}
\begin{theorem} (Classification of $\frac{n(n+1)}{2} - 1$ dimensional upper-triangular regular subalgebras) There are $2n-2$ regular subalgebras of dimension $\frac{n(n+1)}{2} - 1$ in $\mathfrak{sl}(n, \mathbb{C})$ spanned by elements of $E$ and $H$.
\\There are $n-1$ subalgebras without $E_{i,i+1}$.
$$L_{i,i+1} = \s{((E \setminus \{E_{i,i+1})\cup H\}}$$
for $1 \leq i \leq n -1$. \\ There are $n-1$ subalgebras without $H_i$: $$L_i = \s{(E \cup (H - H_i)\}}$$
for $1 \leq i \leq n -1$. The $\{L_i: 1 \leq i \leq n-1\}$ and $\{L_{i,i+1}: 1 \leq i \leq n-1\}$ are in separate conjugacy classes.
\label{full classification}
\end{theorem}
\begin{proof} As $\frac{n(n+1)}{2} = |E| + |H|$, obtaining an $\frac{n(n+1)}{2} - 1$ upper-triangular regular subalgebra means we remove one of the elements from $E \cup H$ and check whether it is still an algebra. We can remove any $H_i$ and maintain closure under the Lie bracket by lemma \ref{remove h_i}. There are $n-1$ such $H_i$'s, therefore, each $L_i$ is a subalgebra. To remove $E_{ij}$ from $E$ we must remove at least $j - i$ elements to maintain closure by lemma \ref{lemma 1}. As we can only remove one element, this means we need $j - i = 1 \implies j = i + 1$. These are precisely the elements on the off-diagonal, and so $L_{i,i+1}$ are the only subalgebras obtained by removing an $E_{i,j}$, and there are $n -1$ of these. Therefore, we have $2n-2$  $\frac{n(n+1)}{2} - 1$ dimensional upper-triangular regular algebras. Finally, as $L_i$ and $L_{i,i+1}$ have differing nilpotent dimension, they are in separate conjugacy classes by corollary \ref{nilpotent dimension} ($L_i$ has nilpotent dimension $\frac{n(n-1)}{2}$ while $L_{i,i+1}$ has nilpotent dimension $\frac{n(n-1)}{2}-1$).
\end{proof}

\begin{lemma}
    Let $T_{i,i+1}^{[1]}$ denote the star matrix of the maximal nilpotent subalgebra of $L_{i,i+1}$. Let $v = {\C^n}_0$. Fix $i$. For all $j < i$, ${T_{j,j+1}}^{n-i} v = {\C^n}_i$ and ${T_{i,i+1}}^{n-i} v = {\C^n}_{i-1}$. A bit more abstractly, the ${n - i}^{th}$ term in the derived algebra of $T_{i,i+1}$, i.e., ${T_{i,i+1}}^{n-i}$'s left action on $\C^{n}$, sends it to $\C^{i - 1}$, an $i - 1$ dimensional subspace, and the ${n - i}^{th}$ term in the derived algebra of $T_{j,j+1}$, i.e., ${T_{j,j+1}}^{n-i}$ left action on $\C^{n}$, sends it to $\C^{i}$, an $i$ dimensional subspace. 
    \label{Nilpotent inductive classification}
\end{lemma}
    \begin{proof}
     Induction on $i$. For $i = n-1$, $T_{n-1, n}v  = {\C^n}_{n-2}$. For $j < n - 1$, $T_{j, j+1} ={\C^n}_{n-1}$. So, $T_{n-1,n}$ sends $\C^{n}$ to ${\C^{n}}_{n-2}$ and for all $j < n - 1$, $T_{j,j+1}$ sends $\C^n$ to ${\C^{n}}_{n-1}$.  Suppose claim holds for $i \leq n - 1$. Then: ${T_{i,i+1}}^{n- i}v = {\C^n}_{i-1}$, where ${T_{i,i+1}}^{n- i}$ sends $\C^n$ to ${\C^{n}}_{i - 1}$ and that for $j < i$,  $ T_{j,j+1}^{n- i}v = {\C^n}_{i}$ where  ${T_{j,j+1}}^{n-i}$ sends $\C^n$ to ${\C^n}_{i}$. In particular, ${T_{i -1, i}}^{n -i}v = {\C^n}_{i}$. Then, ${T_{i -1, i}}^{n -(i - 1)} v = {\C^n}_{i-2}$. We get two more zeros as the ${i-1}^{th}$ and $i^{th}$ rows have a zero in the first $i$ columns.  For $j < i - 1$, ${T_{j, j+1}}^{n-(i-1)} v = T_{j, j+1} ({\C^n}_i) = {\C^n}_{i-1}$. Thus, ${T_{i-1,i}}^{n-(i-1)}$ sends $\C^n$ to ${\C^n}_{i - 2}$ and for $j < i - 1$,  ${T_{j-1,j}}^{n-(i-1)}$ sends $\C^n$ to ${\C^n}_{i - 1}$. The result follows by induction.    
\end{proof}

\begin{theorem} (Complete Classification of $\frac{n(n-1)}{2} - 1$ nilpotent upper-triangular regular subalgebras)
    Let the maximal nilpotent subalgebra of each $L_{i,i+1}$ be denoted by $N_{i,i+1} = \{E_{ij} \in L_{i,i+1}\}$. Then $\{N_{i,i+1}\}$ are all in distinct conjugacy classes.
\label{N_i theorem}
\end{theorem}
\begin{proof}
Let the nilpotent subalgebra of each $L_{i,i+1}$ be denoted by $N_{i,i+1} = \{E_{ij} \in L_{i,i+1}\}$. We can represent them as star matrices:
There is a 0 at the $i,i +1$th position, and stars everywhere in the strictly upper triangular part. As nilpotent elements conjugate to nilpotent elements, we need only show that none of these maximal nilpotent subalgebras conjugate to each other, and so the $\{L_{i,i+1}\}$ are in different conjugacy classes. By corollary \ref{Derived Series Left and row rank} we know that the left and row rank of the derived series is invariant under conjugation, and so by lemma \ref{Nilpotent inductive classification} we have that the each $N_{i,i+1}$ is not conjugate to any $N_{j,j+1}$ for $j < i$, and so inductively, all of the $\{N_{i,i+1}\}$ are in separate conjugacy classes. 
\end{proof}

\begin{corollary}
    $\{L_{i,i+1}\}$ defined in Theorem \ref{full classification} are in different conjugacy classes.
    \label{L_i lemma}
\end{corollary}
\begin{proof}
    By lemma \ref{nilpotent conjugacy} we know that the maximal nilpotent Lie algebras of conjugate algebras are conjugate, but by theorem \ref{N_i theorem}, none of the $N_{i,i+1}$ algebras are conjugate. Therefore,  $\{L_{i,i+1}\}$ are in different conjugacy classes.
\end{proof}

\begin{lemma} $\{L_{i}\}$ defined in Theorem \ref{full classification} are in separate conjugacy classes. Thus, they are in singleton classes.
\label{conjugacy for h_i}
\end{lemma}
\begin{proof}
The result follows from theorem \ref{conjugacy of l_ab}.
\end{proof}
\begin{theorem} (Conjugacy Classes of $\frac{n(n+1)}{2} - 1$ dimensional upper-triangular regular subalgebras) All $\frac{n(n+1)}{2} - 1$ dimensional subalgebras are in distinct conjugacy classes. They are as follows:\\
There are $n-1$ subalgebras without $E_{i,i+1}$.
$$L_{i,i+1} = \s(E \setminus \{E_{i,i+1}\}\cup H)$$
for $1 \leq i \leq n -1$.\\
There are $n-1$ subalgebras without $H_i$.
$$L_i = \s (E \cup H \setminus \{H_i\})$$
for $1 \leq i \leq n -1$.
\label{complete class of one less than full}
\end{theorem}

\subsection{Classification of Codimension 2 subalgebras}
\begin{theorem} 
(Classification of $\frac{n(n+1)}{2} - 2$ dimensional upper-triangular regular subalgebras) There are $2n^2 - 3n - 1$ upper-triangular regular subalgebras of dimension $\frac{n(n+1)}{2} - 2$ in $\mathfrak{sl}(n, \mathbb{C})$. The algebras are of the form:\\
\textbf{Without 2 $H_i$:} There are $\frac{(n-1)(n-2)}{2}$ of these, and they are of the form:
$$P_{ij} = \s(E \cup H \setminus \{H_i, H_j\})$$ where $i, j \in \{1,...,n-1\}$.\\
\textbf{Without $E_{ij}$ and $H_i$:} There are $(n-1)(n-1)$ of these, and they are of the form:
$$M_{ij} = \s((E \setminus \{E_{i,i+1}\}) \cup (H \setminus \{H_j\}))$$ where $i, j \in \{1,...,n-1\}$ and $i < j$.\\
\textbf{Without 2 $E_{ij}$:} There are $\frac{n^2 + n - 6}{2}$ of these, and they are two kinds.\\
There are $\frac{(n-1)(n-2)}{2}$ of these:
$$N_{ij} = \s(E \setminus \{E_{i,i+1}, E_{j,j+1}\} \cup H)$$ where $i, j \in \{1,...,n-1\}$. We choose the ordering $i < j$ to make it consistent.\\
There are $2(n-2)$ of these:
$$N_{R_{i}} = \s((E \setminus \{E_{i,i+2}, E_{i,i+1}\})\cup H)$$
$$N_{C_{i}} = \s((E \setminus \{E_{i,i+2}, E_{i+1, i+2}\})\cup H)$$
for $1 \leq i \leq n - 2$.\\
Additionally, none of the $P, N, M$ subalgebras are conjugate to each other.
\label{subalgebras of 2 dimension less}
\end{theorem}
\begin{proof}
As the algebras are codimension 2, there are three options; remove two $E_{ij}$'s, an $E_{ij}$ and $H_i$, or two $H_i$'s.  
By lemma \ref{remove h_i}, we know that we can remove any number of $H_i$'s and still be a subalgebra. As the order of removal doesn't matter, the total number of these algebras are ${{n-1} \choose 2 }= \frac{n^2-3n+2}{2}$, and so the $P_{ij}$'s as described are subalgebras. 
If we remove one $E_{ij}$, we need $j -i = 1 \implies j = i +1$. Therefore, can only remove $E_{i,i+1}$ or  any $H_i$. As there are $n-1$ of each, we have $(n-1)(n-1)$ of $M_{ij}$'s as described above.
Finally, we can take out two $E_{ij}$'s. So either we take out 2 independent $E_{ij}$'s that don't require the removal of other elements, or we take out 1 $E_{ij}$ which requires the removal of one other element. In the first case, as seen earLier, we can only take out $E_{i,i+1}$ and $E_{j,j+1}$ independently (where $i \neq j$), and there are $\frac{(n-1)(n-2)}{2}$ such pairs. In the second case, we have $j - i = 2 \implies j = i + 2$. So, to preserve the Lie bracket we must either remove $E_{i,i+1}$ or $E_{i+1,i+2}$. Thus, there are $2(n-1)$ such possible combinations. These are exactly the subalgebras described above. So in total there are $2n^2 - 3n - 1$ such subalgebras.
Clearly, the $P_{ij}$'s, \{$N_{ij}, N_{R_{i}}, N_{C_{i}}\}$'s, and the $M_{ij}'s$ are in different conjugacy classes as nilpotent elements conjugate to nilpotent elements by lemma \ref{nilpotent dimension}, and they each have a different number of nilpotent elements. 
\end{proof}

\textbf{Note:} Let $\mathcal{N}_{ij}, \mathcal{N}_{R_{k}}, \mathcal{N}_{C_{l}} $ denote the maximal nilpotent subalgebras of $\{N_{ij}, N_{R_{k}}, N_{C_{l}}\}$. We shall be distinguishing their conjugacy classes in the next four lemmas. All the proofs are analagous to theorem \ref{Nilpotent inductive classification}. We abuse notation and refer to both the star matrix and the subalgebra as $\N_{i,j}$. We do so as the powers of the matrix exactly correspond to the derived series, and the notation for both is the same (by lemma \ref{multiplication by star matrix}).
\begin{lemma} (Classification of $\N_{ij}$)
    Let $v = {\C^n}_0$. Take the maximal nilpotent subalgebra $\N_{i,j}$ of $N_{i,j}$. We claim:\\
    (i) For $l < j$ we have ${\N_{k,l}}^{n-j} v = {\C^n}_j$ with $n - j$ zeros, and ${\N_{i,j}}^{n-j} v = {\C^n}_{j-1}$ with $n - j + 1$ zeros. \\
    A bit more abstractly, the left action of ${\N_{i,j}}^{n-j}$ on $\C^{n}$ sends it to ${\C^n}_{j - 1}$, a ($j - 1$)-dimensional subspace, and the ${n - j}^{th}$ term in the derived algebra of $N_{k,l}$, i.e., the left action of ${N_{k,l}}^{n-j}$ on $\C^{n}$, sends it to ${\C^n}_{j}$ a $j$-dimensional subspace.\\
    (ii) For $i < k$ we have $v^T{\N_{k,l}}^{i} = \begin{bmatrix}  0_1 &  \hdots & 0_{i} & \ast_{i+1} & \hdots & \ast_n   \end{bmatrix}$ with $i - 1$ zeros, and $v^T{\N_{i,j}}^{i}  = \begin{bmatrix}
        0_1 &  \hdots & 0_{i+1} & \ast_{i+1} & \hdots & \ast_n  
    \end{bmatrix}$ has $i + 1$ zeros. \\
    The ${i}^{th}$ term in the derived algebra of $\N_{i,j}$, i.e., the right action of ${\N_{i,j}}^{i}$ on $\C^{n}$, sends it to ${\C^n}_{n - (i+1)}$ an $(n - (i+1))$-dimensional subspace, and the ${i}^{th}$ term in the derived algebra of $N_{k,l}$, i.e., the right action of ${N_{k,l}}^{i}$ on $\C^{n}$ sends it to an $(n-i)$-dimensional subspace.  
    \label{N-ij classification}
\end{lemma}

\begin{lemma} (Classification of $\N_{R_i}$)
    \label{R_i conjugation}
 Let $v = {\C^n}_0$. Let $ 1 \leq i \leq n - 2$ Take the maximal nilpotent subalgebra $\N_{R_i}$ of $N_{R_i}$. We claim
    For $i > j$ we have that ${\N_{R_i}}^{n-(i+1)} v = {\C^n}_i$ with $n - i$ zeros, and ${\N_{R_j}}^{n-(i+1)} v = {\C^n}_{i+1}$ has $n - (i + 1)$ zeros. \\
    A bit more abstractly, the left action of ${\N_{R_i}}^{n-(i+1)}$ on $\C^{n}$ sends it to $\C^{i}$, an $i$-dimensional subspace, and the ${n - (i+1)}^{th}$ term in the derived algebra of $N_{R_j}$, i.e., the left action of ${N_{R_j}}^{n-(i+1)}$ on $\C^{n}$ sends it to ${\C^n}_{i+1}$, an $(i +1)$-dimensional subspace.    
\end{lemma}

\begin{lemma} (Classification of $\N_{C_i}$)
    \label{C_i conjugation}
    Let $v^T = \begin{bmatrix}
        \ast_1 & \hdots & \ast_n
    \end{bmatrix}$ where the entries $\ast_i \in \C$. This represents $\C^n$. Let $1 \leq i \leq n -2$. For $i < k$  $v^T{\N_{C_i}}^{i} = \begin{bmatrix}  0_1 &  \hdots & 0_{i} & \ast_{i+1} & \hdots & \ast_n   \end{bmatrix}$ with $i - 1$ zeros, and $v^T{\N_{i,j}}^{i}  = \begin{bmatrix}
        0_1 &  \hdots & 0_{i+1} & \ast_{i+1} & \hdots & \ast_n  
    \end{bmatrix}$ has $i + 1$ zeros. 
    A bit more abstractly, the ${i}^{th}$ term in the derived algebra of $\N_{i,j}$, i.e., the right action of ${\N_{i,j}}^{i}$ on $\C^{n}$ sends it to $\C^{n - (i+1)}$, an $(n - (i+1))$-dimensional subspace, and the ${i}^{th}$ term in the derived algebra of $N_{k,l}$, i.e., the right action of ${N_{k,l}}^{i}$ on $\C^{n}$ sends it to $\C^{n - i}$, an $(n-i)$-dimensional subspace.  
\end{lemma}

\begin{lemma} (Complete Classification of Nilpotent regular subalgebras of dimension $\frac{n(n-1)}{2}-2$) We claim that the conjugacy classes are $\{\N_{i,i+1}, \N_{R_{i}}, \N_{C_{i}}: 1 \leq i \leq n -2\}, \{\N_{n-1, n}\}$, and all the $\{\N_{ij}\}$ for $j \neq i + 1$  are all in singleton conjugacy classes.
\label{nilpotent classification 2 dimension less}
\end{lemma}
\begin{proof}
We have:
\begin{equation*}
    \N_{ij} = \text{span}(E \setminus \{E_{i,i+1}, E_{j,j+1}\})
\end{equation*}
\begin{equation*}
    \N_{R_i} = \text{span}(E \setminus \{E_{i,i+1}, E_{i,i+2}\})
\end{equation*}
\begin{equation*}
    \N_{C_i} = \text{span}(E \setminus \{E_{i,i+2}, E_{i+1,i+2}\})
\end{equation*}
We abuse notation by denoting both the algebra and its star matrix by the same notation. 
\begin{enumerate}
    \item \textbf{$\N_{ij}$ is not conjugate to $\N_{kl}$ for $(i,j) \neq (k,l)$}: By lemma \ref{acting on column vectors lemma} and \ref{acting on row vectors lemma} we know that two conjugate subalgebras have the same action on column vectors, but by lemma \ref{N-ij classification} we know that this is not the case. Thus, these algebras are not conjugate.
    \item \textbf{$\{\N_{R_i}, \N_{C_i}, \N_{i,i+1}\}$ form a conjugacy class:} $\tau_{i+1, i+2}$ conjugates $N_{C_i}$ to $N_{i,i+1}$. Similarly, $\tau_{i, i+1}$ conjugates $N_{i,i+1}$ to $N_{R_i}$. The product $\tau_{(i+1, i+2)}\tau_{(i,i+1)} = \tau_{(i, i+2)}$ conjugates $\N_{C_i}$ to $\N_{R_i}$. 
\end{enumerate}
$N_{C_i}$ is not conjugate to $N_{C_j}$ otherwise $\N_{i,i+1}$ would be conjugate to $\N_{j,j+1}$. Similarly for $N_{R_i}$. Thus, the conjugacy classes are :$\{\N_{i,i+1}, \N_{R_{i}}, \N_{C_{i}}: 1 \leq i \leq n -2\}, \{\N_{n-1, n}\}$, and $\{\N_{ij}\}$ for $j \neq i$ (the $\{\N_{ij}\}$  for $j \neq i + 1$ are in singleton conjugacy classes.)
\end{proof}

\begin{lemma} (Classification of $P_{ij}$)
$P_{ij}$ is conjugate to $P_{kl}$ iff $(i,j) = (k,l)$.
\label{classification of p_ij}
\end{lemma}

\begin{lemma}(Classification of $M_{ij}$)
$M_{ij}$ is conjugate to $M_{kl}$ iff $(i,j) = (k,l)$.
\label{classification of m_ij}
\end{lemma}
\begin{proof}
$M_{ij}$ and $M_{kl}$ can only be conjugate if $i = k$, as none of the nilpotent subalgebras are conjugate to each other by theorem \ref{complete class of one less than full}. For a fixed $i$ no $M_{ij}$ is conjugate to $M_{il}$ by theorem \ref{conjugacy of l_ab}.
\end{proof}
    \begin{theorem}   
    (Classification of upper-triangular regular subalgebras of dimension $\frac{n(n+1)}{2} - 2$) \\
    The $P_{ij}$ are in singleton conjugacy classes. The $M_{ij}$ are in singleton conjugacy classes.
 $\{N_{ij}, N_{R_{k}}, N_{C_{l}}\}$   $\{N_{i,i+1}, N_{R_{i}}, N_{C_{i}}: 1 \leq i \leq n -2\}, \{N_{n-1, n}\}\}$, and all the $\{N_{ij}\}$ for $j \neq i + 1$  are all in singleton conjugacy classes.
\label{complete class of subalgebras of 2 dimension less}
\end{theorem}
\subsection{Classification of 2 Dimensional regular Algebras}
\begin{theorem}
    The 2 dimensional upper-triangular regular subalgebras of $\mathfrak{sl}(n, \mathbb C)$ spanned by $E$ and $H$ are as follows:    
    \begin{align*}
        A_{(i,j), (k,l)} = \text{span}(\{E_{ij}, E_{kl}\}) & \text{ where $i \neq k,l$ and $j \neq k,l$.} & \text{There are $\frac{n^4 - 2n^3 -n^2 + 2n}{8}$ such algebras.}\\
        A_{(i,j), (i, l)} = \text{span}(\{E_{ij}, E_{il}\}) & \text{ where $i < j < l$.} &\text{There are $\frac{n^3 - 3n^2 + 2n}{6}$ such algebras.}\\
        A_{(i,j), (k,j)} = \text{span}(\{E_{ij}, E_{kj}\}) & \text{ where $i < k < j = l$.} & \text{There are $\frac{n^3 - 3n^2 + 2n}{6}$ such algebras.} \\
        B_{(i,j), k} = \text{span}(\{E_{ij}, H_k\}) & \text{ where $i,j \neq k, k + 1$} & \text{There are $\frac{n^3 - 6n^2 + 11n - 6}{2}$ such algebras.}\\
        B_{(i,j), i} = \text{span}(\{E_{ij}, H_i\}) & \text{ where $j \neq i + 1$}. & \text{There are $\frac{n^2 - 3n + 2}{2}$ such algebras.}\\
        B_{(i+ 1,j), i} = \text{span}(\{E_{i+1, j}, H_i\}) & & \text{There are $\frac{n^2 - 3n + 2}{2}$ such algebras.}\\
         B_{(i,i + 1), i} = \text{span}(\{E_{i, i + 1}, H_i\})& & \text{There are $n-1$ such algebras.}\\
           B_{(i,k), k} = \text{span}(\{E_{i, k}, H_k\}) & \text{ where $i \neq k$.} & \text{There are $\frac{n^2 - 3n + 2}{2}$ such algebras.}\\
           B_{(i,k + 1), k} = \text{span}(\{E_{i, k+1 }, H_k\})& \text{ where $i \neq k$.} & \text{There are $\frac{n^2 - 3n + 2}{2}$ such algebras.}\\
           C_{k,l} = \s(\{H_k, H_l\}) & \text{ where $ k < l$ and $l \neq k + 1$.} & \text{There are $\frac{n^2 - 5n + 6}{2}$ such algebras.}\\
           C_{l,l + 1} = \s(\{H_l, H_{l+1}\}) &  & \text{There are n-2 such algebras}.
    \end{align*}
None of the $A_i$ are conjugate to the $B_j$ and $C_k$, and $B_j$ and $C_k$ are also not conjugate to each other.
\end{theorem}
\begin{proof}
    These are clearly all of the subalgebras, and none of $A_i$ $B_j$ or $C_k$ are conjugate because of lemma \ref{nilpotent dimension}.
\end{proof}
\begin{definition} Let
    \begin{align*}
        A_1 = \{A_{(i,j), (k,l)}: 2\leq i \neq k,l < j \neq k,l \leq n - 1\}\\
        A_2 = \{ A_{(i,j), (i, l)}: 1\leq i < j < l \leq n - 1\} \\
        A_3 = \{A_{(i,j), (k,j)}: 1 \leq i < k < j \leq n - 1\} \\
        B_1 = \{B_{(i,j), k}: 1 \leq i,j \neq k, k + 1 \leq n - 1\}\\
        B_2 = \{B_{(i,j), i}, B_{(i+ 1,j), i} 1\leq i< < i + 1< j \leq n - 1\}\\
      B_3 = \{ B_{(i,i + 1), i}: 1 \leq i < i + 1\leq n - 1\}\\
        B_4 = \{ B_{(i,k), k}, B_{(i,k + 1), k}: 1 \leq i < k < k + 1 \leq n - 1  \}\\
        C_1 = \{C_{k,l}: 1 \leq  k < l \leq n - 1\}
    \\
       C_2 = \{C_{l,l + 1}: 1 \leq l < l + 1 \leq n - 1\}
    \end{align*}
\end{definition}

\begin{lemma}
    \label{A_i are not conjugate}
    $A_1$ $A_2$ and $A_3$ are in separate conjugacy classes. 
\end{lemma}
\begin{proof}
First we show the algebras within the $A_i$'s are conjugate.
\begin{enumerate}
    \item $A_1:$ Let $Q_1 = \s(\{E_{ij}, E_{kl}\})$ and $Q_2 = \s(\{E_{mn}, E_{st}\}) \in A_1$. $P = \tau_{im}\tau_{jn}\tau_{ks} \tau_{lt}$ conjugates $Q_1$ to $Q_2$.
    \item $A_2:$ Let $Q_1 = \s(\{E_{ij}, E_{il}\})$ and $Q_2 = \s(\{E_{km}, E_{kn}\}) \in A_2$. $P = \tau_{ik}\tau_{jm}\tau_{ln}$ conjugates $Q_1$ to $Q_2$.
    \item $A_3$: Let $Q_1 = \s(\{E_{ij}, E_{kj}\})$ and $Q_2 = \s(\{E_{mn}, E_{sn}\}) \in A_2$. $P = \tau_{im}\tau_{ks}\tau_{jn}$ conjugates $Q_1$ to $Q_2$. 
\end{enumerate}

\begin{table}[ht]
\centering
\caption{Left and Right Action of $A_i$}
\begin{tabular}{||c| c |c||} 
 \hline
 Conjugacy Class & Left Action & Right Action \\ [0.5ex] 
 \hline\hline
 $A_1$ & $\begin{bmatrix}
    & & & & \\
    & & \ast_{ij} & & \\
    & & & &\\
    & & & & \ast_{mn}\\
 \end{bmatrix} \begin{bmatrix}
     \ast_1 \\ \vdots \\ \ast_n
 \end{bmatrix} = \begin{bmatrix}
    \\ \ast_{i} \\ \\ \ast_{m} \\
 \end{bmatrix}$ &  $ \begin{bmatrix}
     \ast_1 & \hdots & \ast_n
 \end{bmatrix} \begin{bmatrix}
    & & & & \\
    & & \ast_{ij} & & \\
    & & & &\\
    & & & & \ast_{mn}\\
 \end{bmatrix}= \begin{bmatrix}
    & \ast_{j} & & \ast_{n} & 
 \end{bmatrix}$ \\ 
 \hline
 $A_2$ & $\begin{bmatrix}
    & & & & \\
    & & \ast_{ij} &  &\ast_{il} \\
    & & & &\\
    & & & & \\
 \end{bmatrix} \begin{bmatrix}
     \ast_1 \\ \vdots \\ \ast_n
 \end{bmatrix} = \begin{bmatrix}
    \\ \ast_{i} \\ \\ \\
 \end{bmatrix}$ &  $ \begin{bmatrix}
     \ast_1 & \hdots & \ast_n
 \end{bmatrix} \begin{bmatrix}
    & & & & \\
    & & \ast_{ij} &  &\ast_{il} \\
    & & & &\\
    & & & & \\
 \end{bmatrix} = \begin{bmatrix}
    & \ast_{j} & & \ast_{l} & 
 \end{bmatrix}$ \\
   \hline
 $A_3$ & $\begin{bmatrix}
    & & & & \\
    & & \ast_{ij} & & \\
    & & & &\\
    & & \ast_{mj}& & \\
 \end{bmatrix} \begin{bmatrix}
     \ast_1 \\ \vdots \\ \ast_n
 \end{bmatrix} = \begin{bmatrix}
    \\ \ast_{i} \\ \\ \ast_{m} \\
 \end{bmatrix}$ &  $ \begin{bmatrix}
     \ast_1 & \hdots & \ast_n
 \end{bmatrix} \begin{bmatrix}
    & & & & \\
    & & \ast_{ij} & & \\
    & & & &\\
    & & \ast_{mj}& & \\
 \end{bmatrix}= \begin{bmatrix}
    & \ast_{j} & & & 
 \end{bmatrix}$ \\
[1ex] 
 \hline
\end{tabular}
  \label{A_i conjugation}
\end{table}
By lemmas \ref{acting on column vectors lemma} and \ref{acting on row vectors lemma}, $A_1, A_2,$ and $A_3$ are in distinct conjugacy classes.
\end{proof}

\begin{lemma}
$B_1$ $B_2, B_3$ and $B_4$ are separate conjugacy classes. 
    \label{B_i are not conjugate}
\end{lemma}
\begin{proof} We note:

\begin{enumerate}
    \item $B_1$: Let $Q_1 = \s(\{E_{ij}, H_k\})$ and $Q_2 = \s(\{E_{mn}, H_l\}) \in B_1$. $P = \tau_{im}\tau_{jn}\tau_{kl}$ conjugates $Q_1$ to $Q_2$.
    \item $B_2$: Let $Q_1 = \s(\{E_{ij}, H_i\})$ and $Q_2 = \s(\{E_{mn}, H_m\}) \in B_2$. $P_1 = \tau_{im}\tau_{jn}\tau_{i+1, m+1}$ conjugates $Q_1$ to $Q_2$. Similarly, let $R_1 = \s(\{E_{k+1, j}, H_k \})$ and $R_2 = \s(\{E_{m+1, l}, H_m\} )\in B_2$.  $P_2 = \tau_{k+1, m+1}\tau_{km}\tau_{jl}$ conjugates $R_1$ to $R_2$. Finally, let $S_1 = \s(\{E_{kj}, H_k\})$ and $S_2 = \s(\{E_{k+1,j}, H_k\}) \in B_2$. $P_3 = \tau_{k,k+1}$ conjugates $S_1$ to $S_2$. Thus, every algebra in $B_2$ is conjugate to each other (we take the same $j$ in the final case as for differing j's, those algebras are already conjugate to each other by $P_1$ or $P_2$.)
    \item $B_3:$ Let $Q_1 = \s(\{E_{k,k+1}, H_k\})$ and $Q_2 = \s(\{E_{i,i+1}, H_i\}) \in B_3$. $P_1 = \tau_{ki}\tau_{k+1, i+1}$ conjugates $Q_1$ to $Q_2$.
    \item $B_4:$ Let $Q_1 = \s(\{E_{ij}, H_j\})$ and $Q_2 = \s(\{E_{mn}, H_n\}) \in B_4$. $P_1 = \tau_{im}\tau_{jn}\tau_{j+1, n+1}$ conjugates $Q_1$ to $Q_2$. Similarly, let $R_1 =\s( \{E_{i,k+1}, H_k \})$ and $R_2 = \s( \{E_{m, n+1}, H_n\}) \in B_4$. $P_2 = \tau_{k+1, n+1}\tau_{im}\tau_{kn}$ conjugates $R_1$ to $R_2$. Finally, let $S_1 = \s(\{E_{ik}, H_k\}) $ and $S_2 = \s(\{E_{i,k+1}, H_k\}) \in B_2$. $P_3 = \tau_{k,k+1}$ conjugates $S_1$ to $S_2$. Thus, every algebra in $B_2$ is conjugate to each other (we take the same $i$ in the final case as for differing i's, those algebras are already conjugate to each other by $P_1$ or $P_2$.) $B_1$ has a different nilpotent dimension from $B_2, B_3$, and $B_4$. Therefore by lemma \ref{nilpotent dimension}, $B_1$ is in a separate conjugacy class.
\end{enumerate}
  \begin{table}[ht]
\centering
\caption{Left and Right Action of $B_i$}
\begin{tabular}{||c| c |c |c||} 
 \hline
 Class & Algebra & Left Action & Right Action \\ [0.5ex] 
 \hline\hline
 $B_2$ & $\{E_{ij}, H_i\}$ &$\begin{bmatrix} \\
    & & \lambda_i &  &\ast_{ij} & \\
    &  & &-\lambda_i  & \\
    \\
 \end{bmatrix} \begin{bmatrix}
     \ast_1 \\ \vdots \\ \ast_n
 \end{bmatrix} = \begin{bmatrix}
    \\ \ast_{i} \\ \ast_{i+1} \\
 \end{bmatrix}$ &  $ \begin{bmatrix}
     \ast_1 & \hdots & \ast_n
 \end{bmatrix} \begin{bmatrix} \\
    & & \lambda_i & &\ast_{ij} & \\
    &  & && -\lambda_i \\
    \\ \end{bmatrix}= \begin{bmatrix}
    & \ast_{i} & \ast_{i+1}& & \ast_{j} 
 \end{bmatrix}$ \\ 
 \hline
 $B_3$ & $\{E_{i,i+1}, H_i\}$ &$\begin{bmatrix} \\
    & & \lambda_i &  \ast_{i,i+1} & \\
    &  & &-\lambda_i  & \\
    \\
 \end{bmatrix} \begin{bmatrix}
     \ast_1 \\ \vdots \\ \ast_n
 \end{bmatrix} = \begin{bmatrix}
    \\ \ast_{i} \\ \ast_{i+1} \\
 \end{bmatrix}$ &  $ \begin{bmatrix}
     \ast_1 & \hdots & \ast_n
 \end{bmatrix} \begin{bmatrix} \\
    & & \lambda_i &  \ast_{i,i+1} & \\
    &  & &-\lambda_i  & \\
    \\ \end{bmatrix}= \begin{bmatrix}
    & \ast_{i} & \ast_{i+1}&  
 \end{bmatrix}$ \\ 
   \hline
 $B_4$ & $\{E_{m,i}, H_i\}$ &$\begin{bmatrix} \\
    & & \ast_{m,i}\\
    & & \lambda_i &   & \\
    &  & &-\lambda_i  & \\
    \\
 \end{bmatrix} \begin{bmatrix}
     \ast_1 \\ \vdots \\ \ast_n
 \end{bmatrix} = \begin{bmatrix}
    \\ \ast_m \\ \\ \ast_{i} \\ \ast_{i+1} 
 \end{bmatrix}$ &  $ \begin{bmatrix}
     \ast_1 & \hdots & \ast_n
 \end{bmatrix} \begin{bmatrix} \\
    & & \ast_{m,i}\\
    & & \lambda_i &   & \\
    &  & &-\lambda_i  & \\
    \\ \end{bmatrix} = \begin{bmatrix}
    & \ast_{i} & \ast_{i+1}&  
 \end{bmatrix}$ \\ 
[1ex] 
 \hline
\end{tabular}
  \label{B_i conjugation}
\end{table}
By lemmas \ref{acting on column vectors lemma} and \ref{acting on row vectors lemma}, $B_2, B_3,$ and $B_4$ are in separate conjugacy classes. Thus, all $B_i$ are distinct.
\end{proof}

\begin{lemma}
$C_1$ and $C_2$ are in separate conjugacy classes.
    \label{C_i are not conjugate}
\end{lemma}
\begin{proof}
First we show the algebras within the $C_{i}$'s are conjugate.
\begin{enumerate}
    \item $C_1$: Let $Q_1 = \{H_k, H_l\}$ and $Q_2 = \{H_m, H_n\} \in C_1$. $P = \tau_{km}\tau_{k+1, m+1}\tau_{ln} \tau_{l+1, n+1}$ conjugates $Q_1$ to $Q_2$.
    \item $C_2:$ Let $Q_1 = \{H_l, H_{l+1}\}$ and $Q_2 = \{H_m, H_{m+1}\} \in C_2$. $P = \tau_{lm} \tau_{l+1, m+1}$ conjugates $Q_1$ to $Q_2$. 
    \item $C_1 = \begin{bmatrix}
    \\
    & & \lambda_{k} \\
    & & & -\lambda_{k+1}\\
    & & & & & \alpha_{l} \\
    & & & & & & -\alpha_{l+1}
 \end{bmatrix} \begin{bmatrix}
     \ast_1 \\ \vdots \\ \ast_n
 \end{bmatrix} = \begin{bmatrix}
    \\ \ast_{k} \\ \ast_{k+1} \\ \\ \ast_{l} \\ \ast_{l+1}
 \end{bmatrix}$ and $C_2 = \begin{bmatrix}
    \\
    & & \lambda_{k} \\
    & & & \alpha_{k+1}-\lambda_{k+1}\\
    & & & & - \alpha_{k+2} \\
    \\ \end{bmatrix} \begin{bmatrix}
     \ast_1 \\ \vdots \\ \ast_n
 \end{bmatrix} = \begin{bmatrix}
    \\ \ast_{k} \\ \ast_{k+1} \\ \ast_{k+2} \\ 
 \end{bmatrix}$. By lemma \ref{acting on column vectors lemma}, $C_1$ and $C_2$ are not conjugate.
\end{enumerate}
\end{proof}

\begin{theorem} (Classification of 2D subalgebras of $\mathfrak{sl}(n, \mathbb C)$) 
The conjugacy classes of the 2D subalgebras are as follows: $A_1, A_2, A_3, B_1, B_2, B_3, B_4, C_1, C_2$. There are 9 conjugacy classes in total.
\end{theorem}

\subsection{Some More General Results}
Define: the diagonal algebras, $D_i = \s(E \setminus \{E_{i,i+1},..., E_{i+k-1, i+k}\}$, row algebras, $R_j = \s(E \setminus \{E_{j,j+1}, ..., E_{j,j+k}\})$, and column algebras $C_m = \s(E \setminus \{E_{m,m+k}, ..., E_{m+ k - 1,m + k}\})$. We examine the action on row and column vectors. All these proofs follow from the standard induction argument outlined earlier.
\begin{lemma} (Action of $D_i$ on column vectors)
    Let $i < j$. $(D_j)^{n-(j + k - 1)}(\C^n) \cong \C^{j+ k -2}$ and $(D_i)^{n-(j + k -1)}(\C^n) \cong \C^{j + k -1}$
    \label{action of d_i on column vectors}
\end{lemma}

\begin{lemma} (Action of $D_i$ on row vectors)
    Let $i < j$. $(\C^n)(D_j)^{i} \cong \C^{n-i}$ and $(\C^n)(D_i)^{i} \cong \C^{n - (i+1)}$.
    \label{action of d_i on row vectors}
\end{lemma}

\begin{lemma} (Action of $R_j$ on column vectors)
    Let $i < j$. $(R_j)^{n-(j+k-1)}(\C^n) \cong \C^{j + k - 2}$ and $(R_i)^{n-(j+k-1)}(\C^n) \cong \C^{j+k-1}$.
    \label{action of r_j on column vectors}
\end{lemma}

\begin{lemma} (Action of $R_j$ on row vectors)
    Let $i < j$. $(\C^n)(R_j)^{i} \cong \C^{n-i}$ and $(\C^n)(R_i)^{i} \cong \C^{n - (i+1)}$
    \label{action of r_j on row vectors}
\end{lemma}
\begin{lemma} (Action of $C_m$ on column vectors)
    Let $p < m$. $(C_m)^{n-(m+k-1)}(\C^n) \cong \C^{m+ k-2}$ and $(C_p)^{n-(m+k-1)}(\C^n) \cong \C^{m+k-1}$.
    \label{action of c_m on column vectors}
\end{lemma}

\begin{lemma} (Action of $C_m$ on row vectors)
    Let $p < n$. $(\C^n)(C_m)^{p} \cong \C^{n-(p+1)}$ and $(\C^n)(C_m)^{p} \cong \C^{n -p}$.
    \label{action of c_m on row vectors}
\end{lemma}
Two conjugate algebras have the same row and column actions; the only algebras that could be conjugate are $\{D_i, R_i, C_i\}$ for any index $i$.  For $k = 1$ (the trivial case, we only have $D_i$) and for $k = 2$ they are conjugate (theorem \ref{complete class of subalgebras of 2 dimension less}). Suppose $k > 2$. 
Let $\mathcal{E} = E \setminus \{E_{i,i+1}: 1 \leq i \leq n -1\}$. The commutators are:
\begin{equation*}
    [D_i, D_i] = \s(\mathcal{E} \setminus \{E_{i-1,i+1},..., E_{i + k,i  k + 2}, E_{i,i+3}, ..., E_{i+k-2, i + k + 1} \})
\end{equation*}
\begin{equation*}
    [R_i, R_i] = \s(\mathcal{E} \setminus \{E_{i-1,i+1}, E_{i,i +2} ..., E_{i,i + k + 1}\})
\end{equation*}
\begin{equation*}
    [C_i,C_i] = \s(\mathcal{E} \setminus \{E_{i-1, i + k}, ..., E_{i+k -2, i + k}, E_{i + k - 1, i + k +1}\})
\end{equation*}
If two algebras $L_1,L_2$ are conjugate then $\dim[L_1,L_2] = \dim[L_2,L_2] $ by lemma \ref{dimension of commutator}. As $\mathcal{E}$ has fixed dimension, it is sufficient to see that the algebras have the same codimension in $\mathcal{E}$.
\begin{center}
\begin{tabular}{||c |c| c |c |c| c|} 
 \hline
 Case No. & Constraints & $|\mathcal{E}| - |D_i|$ & $|\mathcal{E}| - |R_i|$ & $|\mathcal{E}| - |C_i|$ & Equality? \\ [0.5ex] 
 \hline\hline
 1 & $1 \leq i -1$ and $i + k + 2 \leq n$ & $2k$ & $k + 1$ & $k + 1$ & $2k = k + 1 \implies k = 1$\\ 
 \hline
 2 &  $1 \leq i -1$ and $i + k + 1 \leq n < i + k + 2 $ & $2k - 1$ & $k + 1$& $k + 1$ & $2k -1 = k + 1 \implies k = 2$ \\
 \hline
 3 & $1 \leq i - 1$ and $i + k + 1 > n$  & $2k-2$ & $k$ & $k$& $2k -2 = k \implies k = 2$\\
 \hline
 4 & $i - 1 = 0$ and $i + k + 2 \leq n$ & $2k - 1$ & $k$ & $k$ & $2k -1 = k \implies k = 1$\\
 \hline
 5 & $i - 1 = 0$ and $i + k + 2 = n + 1$, $i + k+ 1 = n$ & $2k - 2$ & $k$ & $k$ & $2k -2 = k \implies k = 2$ \\
 \hline 
 6 & $i = 1$ and $i + k + 1 > n$ & $2k - 3$ & $k - 1$ & $k - 1$ & $2k -3 = k - 1 \implies k = 2$\\ [1ex] 
 \hline
\end{tabular}
\end{center}
Therefore, $\dim [D_i, D_i] \neq \dim [R_i, R_i]$ and $\dim [D_i, D_i] \neq \dim [C_i, C_i]$ for $k > 2$. Finally, $C_i$ conjugates to $R_i$ by
\begin{equation*}
    (\tau_{i,i + k} \prod_{m = 1}^{k - 1}\tau_{i + m - 1, i + m}) \cdot C_i \cdot (\tau_{i,i + k} \prod_{m = 1}^{k - 1}\tau_{i + m - 1, i + m} )^{-1}= R_i  
\end{equation*}
Which results in the following classification theorem:
\begin{theorem} (Classification of $D_i$, $R_j$, and $C_m$ type algebras) Let $D_i, R_j$ and $C_m$ be as defined earlier.
For $k \geq 2$
\begin{enumerate}
    \item $D_i$ is conjugate to $D_j$ $\iff i = j$.
    \item $R_j$ is conjugate to $R_i$ $\iff i = j$.
    \item $C_m$ is conjugate to $C_p$ $\iff m = p$.
\end{enumerate}
For $k > 2$, $D_i, R_j, C_m$ are not pairwise conjugate. For $k = 2$, $D_i, R_i, C_i$ are conjugate algebras. For $k = 1$, $D_i = R_i = C_i$, and all $D_i$ are in separate conjugacy classes.   
    \label{classification of d_i, r_j, and C_m}
\end{theorem}

\section{Acknowledgements}
I am deeply grateful to Professor Joe Repka (University of Toronto) for his supervision throughout this project. His insights and comments are what resulted in this paper coming to completion. I am indebted to my amazing parents, Sunil Dhar and Reeta Kantroo Dhar, and my incredible brother, Manik Dhar (MIT) for many useful discussions and sanity checks. I wouldn't be the mathematician I am without their support. I would like to thank all my friends at U of T for their unending support: Aditya Chugh, Nadia Ahsan Gulzar, Franklin Yeung, Harsh Jaluka, Sam Lakerdas-Gayle, Max Sokolov, Amelie Zhang, Leo Watson, Rishibh Prakash, Dharmik Patel, Arkaprava Choudhury, Angel Gao, Guia Janelle Pucyatan, and Alejandra Ceballos.

\end{document}